\documentclass{amsart}
\usepackage{amssymb,amsthm,amsmath,epsfig,latexsym,xypic,supertabular}
\usepackage{calc}
\usepackage{latexsym}
\usepackage{amscd,amssymb,subfigure,hyperref}
\usepackage[arrow,matrix,graph,frame,poly,arc,tips]{xy}
\usepackage{color}

\begin{document}

\newcommand{\mmbox}[1]{\mbox{${#1}$}}
\newcommand{\affine}[1]{\mmbox{{\mathbb A}^{#1}}}
\newcommand{\Ann}[1]{\mmbox{{\rm Ann}({#1})}}
\newcommand{\caps}[3]{\mmbox{{#1}_{#2} \cap \ldots \cap {#1}_{#3}}}
\newcommand{\N}{{\mathbb N}}
\newcommand{\Z}{{\mathbb Z}}
\newcommand{\Q}{{\mathbb Q}}
\newcommand{\R}{{\mathbb R}}
\newcommand{\KK}{{\mathbb K}}
\newcommand{\A}{{\mathcal A}}
\newcommand{\B}{{\mathcal B}}
\newcommand{\OO}{{\mathcal O}}
\newcommand{\C}{{\mathbb C}}
\newcommand{\PP}{{\mathbb P}}
\newcommand{\OS}{{T^d(X,p)}}

\newcommand{\Tor}{\mathop{\rm Tor}\nolimits}
\newcommand{\Ext}{\mathop{\rm Ext}\nolimits}
\newcommand{\Hom}{\mathop{\rm Hom}\nolimits}
\newcommand{\Sym}{\mathop{\rm Sym}\nolimits}
\newcommand{\im}{\mathop{\rm im}\nolimits}
\newcommand{\rk}{\mathop{\rm rk}\nolimits}
\newcommand{\codim}{\mathop{\rm codim}\nolimits}
\newcommand{\supp}{\mathop{\rm supp}\nolimits}
\newcommand{\coker}{\mathop{\rm coker}\nolimits}
\newcommand{\st}{\mathop{\rm st}\nolimits}
\newcommand{\lk}{\mathop{\rm lk}\nolimits}
\sloppy

\newtheorem{thm}{Theorem}[section]
\newtheorem*{thm*}{Theorem}
\newtheorem{defn}[thm]{Definition}
\newtheorem{prop}[thm]{Proposition}
\newtheorem{pref}[thm]{}
\newtheorem*{prop*}{Proposition}
\newtheorem{conj}[thm]{Conjecture}
\newtheorem{lem}[thm]{Lemma}
\newtheorem{rmk}[thm]{Remark}
\newtheorem{cor}[thm]{Corollary}
\newtheorem{notation}[thm]{Notation}
\newtheorem{exm}[thm]{Example}

\newcommand{\msp}{\renewcommand{\arraystretch}{.5}}
\newcommand{\rsp}{\renewcommand{\arraystretch}{1}}

\newenvironment{lmatrix}{\renewcommand{\arraystretch}{.5}\small
  \begin{pmatrix}} {\end{pmatrix}\renewcommand{\arraystretch}{1}}
\newenvironment{llmatrix}{\renewcommand{\arraystretch}{.5}\scriptsize
  \begin{pmatrix}} {\end{pmatrix}\renewcommand{\arraystretch}{1}}
\newenvironment{larray}{\renewcommand{\arraystretch}{.5}\begin{array}}
  {\end{array}\renewcommand{\arraystretch}{1}}

\title[The Weak Lefschetz property for quotients by Quadratic Monomials]
{The Weak Lefschetz property for quotients by Quadratic Monomials}
\dedicatory{Dedicated to the memory of our friend Tony Geramita}
\author{Juan Migliore}
\address{Migliore: Mathematics Department \\ University of
  Notre Dame \\
    South Bend \\ IN 46566\\USA}
\email{migliore.1@nd.edu}

\author{Uwe Nagel}
\address{Nagel: Mathematics Department \\ University of 
  Kentucky \\
    Lexington \\ KY 40506\\USA}
\email{uwe.nagel@uky.edu}

\author{Hal Schenck}
\address{Schenck: Mathematics Department \\ University of
  Illinois \\
    Urbana \\ IL 61801\\USA}
\email{schenck@math.uiuc.edu}


\begin{abstract}
\noindent In \cite{MMR}, Micha\l{}ek--Mir\'o-Roig give a beautiful geometric characterization of Artinian quotients by ideals generated by quadratic or cubic
monomials, such that the multiplication map by a general linear form fails
to be injective in the first nontrivial degree. Their work was 
motivated by conjectures of Ilardi \cite{I} and Mezzetti-Mir\'o-Roig-Ottaviani \cite{MMRO}, connecting the failure to Laplace equations and classical results of Togliatti \cite{T1}, \cite{T2} on osculating planes. We study quotients by quadratic monomial ideals, explaining failure of the Weak Lefschetz Property for some cases not covered by \cite{MMR}.
\end{abstract}

\maketitle

\renewcommand{\thethm}{\thesection.\arabic{thm}}
\setcounter{thm}{0}

\section{Introduction}\label{sec:one}
The Hard Lefschetz Theorem \cite{Hodge} is a landmark result in algebraic
topology and geometry: for a smooth $n$-dimensional projective variety $X$, cup
product with the $k^{th}$ power of the hyperplane class $L$ gives an isomorphism
between $H^{n-k}(X)$ and $H^{n+k}(X)$. A consequence of this is that multiplication by a generic linear form is injective up to degree $n$, and surjective in degree $\ge n$. This result places strong constraints on the Hilbert function, and was used to spectacular effect by Stanley \cite{S2} in proving the necessity conditions in McMullen's conjecture \cite{MCM} on the possible numbers of faces of a
simplicial convex polytope. In short, Lefschetz properties are of
importance in algebra, combinatorics, geometry and topology; their
study is a staple in the investigation of graded Artinian algebras.

\begin{defn}
Let $I =\langle f_0, \ldots, f_k \rangle \subseteq S=\KK[x_0,\ldots, x_r]$ be an ideal with $A=S/I$ Artinian. 
Then $A$ has the {\em Weak Lefschetz Property} $($WLP$)$ if there is 
an $\ell \in S_1$ such that for all $i$, the multiplication map 
$\mu_{\ell}: A_i \longrightarrow A_{i+1}$ has maximal rank. If not, we say that $A$ fails WLP in degree $i$. 
\end{defn}
The set of elements $\ell\in S_1$ with the property that the 
multiplication map $\mu_\ell$ has maximum rank is a (possibly empty) 
Zariski open set in $S_1$. Therefore, the existence of the Lefschetz 
element $\ell$ in the definition of the Weak Lefschetz Property 
guarantees that this set is  nonempty and is equivalent to the 
statement that for a general linear form in $S_1$ the 
corresponding multiplication map has full rank. Throughout this paper,
$A$ will denote a standard graded, commutative Artinian algebra, which is
the quotient of a polynomial ring $S$ as above by a homogeneous ideal $I$. 
The Weak Lefschetz Property depends strongly on char($\KK$). For 
quadratic monomial ideals there is a natural connection to topology, and 
homology with $\mathbb{Z}/2$ coefficients plays a central role in understanding WLP.

Results of \cite{HMNW} show that WLP always holds for $r\le 1$ in characteristic zero.
For $r = 2$ WLP sometimes fails, but is known to hold for many classes:
\begin{itemize}
\item Ideals of general forms \cite{A}
\item Complete intersections \cite{HMNW}
\item Ideals with semistable syzygy bundle \cite{BK}
\item Almost complete intersections with unstable syzygy bundle \cite{BK}
\item Level monomial ideals of type 2 \cite{BMMRNZ}, \cite{CN}
\item Ideals generated by powers of linear forms \cite{SS}, \cite{MMN2}
\item Monomial ideals of small type \cite{CN}, \cite{BMMRNZ}
\end{itemize}
Nevertheless, the $r = 2$ case is still not completely understood, and there remain 
intriguing open questions: for example, does every Gorenstein $A$ have WLP? 
For $r \ge 3$ far less is known; it is open if every $A$ which is a complete intersection has WLP. 
The survey paper \cite{MN} contains many questions and
conjectures. Our basic reference for Lefschetz properties is the book
of Harima-Maeno-Morita-Numata-Wachi-Watanabe \cite{HMMNWW}.
In this paper, we study WLP for the class of quadratic monomial
ideals using tools from homology and topology.


\subsection{Laplace equations and Togliatti systems}

One particularly interesting case occurs when the generators of
$I$ are all of the same degree. In this case, since $V(I)$ is empty,
$I$ defines a basepoint free map
\[
\PP^r \stackrel{\phi_{I}}{\longrightarrow}\PP^k,
\]
and it is natural to ask how WLP connects to the geometry of
$\phi_I$. (See \cite{BMMNZ2} for 
an unexpected  connection between this approach and Hesse configurations, in the first non-trivial case for the Gorenstein problem mentioned above.) First, we need some preliminaries.

\begin{defn}
For an $r$ dimensional variety
$X \subseteq \mathbb{P}^m$ and $p \in X$ such that $\mathcal{O}_{X,p}$ has local defining
equations $f_i$, the $d^{th}$ osculating space $\OS$ is the linear
subspace spanned by $p$ and all $\frac{\partial(f_i)}{\partial x^\alpha}(p)$, with $|\alpha| \le d$.
\end{defn}

At a general point $p \in X$, the expected dimension of $\OS$ is 
\[
\min\{m, {r+d \choose d}-1\}.
\]
If for some positive $\delta$ the dimension at a general point is ${r+d \choose d}-1-\delta < m$, then 
$X$ is said to satisfy $\delta$ Laplace equations of order $d$. 
In \cite{T1}, \cite{T2} Togliatti studied such systems, and showed that
the map $\mathbb{P}^2 \rightarrow \mathbb{P}^5$ defined by 
\[
\{xy^2,yx^2,zy^2,yz^2,xz^2,zx^2\}
\]
is the only projection from the triple Veronese surface to $\mathbb{P}^5$ which
satisfies a Laplace equation of order two. In \cite{I}, Ilardi studied 
rational surfaces of sectional genus one in $\mathbb{P}^5$ satisfying a 
Laplace equation, and conjectured that Togliatti's example generalizes to be
the only map to $\mathbb{P}^{r(r+1)-1}$ via cubic monomials with smooth image.

Mezzetti-Mir\'o-Roig-Ottaviani noticed that the ``missing monomials'' which correspond 
to the coordinate points from which the triple Veronese surface is projected are 
\[
\{x^3,y^3,z^3, xyz\}
\]
whose quotient in $\KK[x,y,z]$ was shown in \cite{BK} to fail WLP. The observation led to the paper \cite{MMRO}, where they prove

\begin{thm}\label{MMROthm}\cite{MMRO}
For a monomial ideal $I=\langle f_0,\ldots,f_k\rangle$ generated in degree $d$
with Artinian quotient $A$, let $X$ denote the variety of the  
image of the map to $\mathbb{P}^k$ defined by $I$, and $X^\perp$ 
the variety of image of the complementary map defined by the monomials
of $S_d \setminus I_d$. Then for $k+1 \le {r+d-1 \choose r-1}$, the
following are equivalent:
\begin{itemize}
\item $\mu_l:A_{d-1}\rightarrow A_d$ fails WLP.
\item $\{ f_0,\ldots,f_k \}$ become linearly dependent in $S/l$ for a generic linear form l.
\item The variety $X^\perp$ satisfies at least one Laplace equation of order $d$.
\end{itemize}
\end{thm}

\begin{rmk}
The numerical hypothesis in Theorem \ref{MMROthm} is equivalent to the condition that $\dim A_{d-1} \leq \dim A_d$; the condition that $\mu_l$ fails WLP is only an assertion that injectivity fails. The classification theorem for injectivity  in Theorem \ref{MMRthm} below does not address surjectivity, and is one motivation for this paper.
\end{rmk}

The varieties $X$ and $X^\perp$ are often called polar varieties, and
can be defined more generally using inverse systems, but we will not need that
here. Mezzetti-Mir\'o-Roig-Ottaviani used Theorem~\ref{MMROthm} 
to produce counterexamples to Ilardi's conjecture. Building on this work, 
Micha\l{}ek and Mir\'o-Roig showed in \cite{MMR} that a suitable modification of Ilardi's conjecture 
is true. A monomial system is {\em Togliatti}  if it
satisfies the equivalent conditions of Theorem~\ref{MMROthm}, {\em smooth} if
$X^\perp$ is smooth, and {\em minimal} if no proper subset satisfies
the equivalent conditions of Theorem~\ref{MMROthm}. For quadrics, Micha\l{}ek-Mir\'o-Roig classify the failure of injectivity in degree one:

\begin{thm}\label{MMRthm}\cite{MMR}
A smooth, monomial, minimial Togliatti system of quadrics $I \subseteq S$ is (after reindexing) $I= \langle x_0,\ldots, x_i\rangle^2 + \langle
x_{i+1},\ldots,x_r\rangle^2$, with $1 \le i \le r-2.$
\end{thm}

Maps defined by monomials of degree $d$ are toric, and a key
ingredient in \cite{MMR} and \cite{MMRO} is Perkinson's classification in \cite{P}
of osculating spaces in terms of the lattice points corresponding to monomials.


\subsection{Results of this paper}
For $\KK[x_0,\ldots,x_5]$ all failures of WLP for quadratic monomial quotients are consequences of Theorems~\ref{MMROthm} and~\ref{MMRthm}, except Examples~\ref{HasWLP} and ~\ref{NoWLP} below and Example~\ref{failfromses}. Failure of WLP for all these cases is explained by our results.

\begin{exm}\label{HasWLP}
Theorem~\ref{MMRthm} gives necessary and sufficient conditions for
failure of WLP for $I$ quadratic if $k+1 \le {r+1 \choose 2}$, since in this range $\mu_l$ cannot be surjective in degree one. For $k+1 > {r+1 \choose 2}$, WLP can still hold if $\mu_l$ is surjective. When $r=5$, $\mu_l$ is
surjective and WLP holds for every Artinian monomial ideal with $k+1 \in \{16, \ldots, 21\}$, except the ideal of all quadratic monomials save (up to relabelling) $\{ x_0x_2, x_0x_3, x_1x_2, x_1x_3\}$, where surjectivity fails from degree 1 to degree 2. 
\end{exm}

\begin{exm}\label{NoWLP} Let $I= \langle x_0^2, x_1^2,x_2^2,x_3^2,x_4^2,x_5^2, x_0x_1, x_2x_3, x_4x_5\rangle
\subseteq \KK[x_0,\ldots, x_5]$; $S/I$ 
has Hilbert series $(1,6,12,8)$, and fails WLP (surjectivity) in degree two.
\end{exm}
By our Proposition~\ref{WLPses}, WLP fails for Example~\ref{HasWLP}, and by our Proposition~\ref{TopFailLift}, WLP fails for Example~\ref{NoWLP}. Failure in both examples also follows from Theorem~\ref{Ttensor}, which generalizes Theorem~\ref{MMRthm}. The following results will be useful:
\begin{prop}\label{injsurj}\cite{MMN}
If $\mu_l: A_i \twoheadrightarrow A_{i+1}$, then it is surjective for
all $j \ge i$. If $A$ is level, and if $\mu_l: A_i \hookrightarrow A_{i+1}$, then it is injective for all $j \le i$.
\end{prop}

\begin{prop}\label{sum}\cite{MMN}
For a monomial ideal $I$, the form $l=\sum_{i=0}^r x_i$ is generic.
\end{prop}

\section{Squarefree monomial ideals}
Let $I$ be a monomial ideal of the form $I = J'+J_\Delta$, where $J'$ is the ideal consisting of squares
of variables, and $J_\Delta$ is a squarefree monomial ideal;
$J_\Delta$ defines a Stanley-Reisner ideal 
corresponding to a simplicial complex $\Delta$. 
\begin{lem}
    \label{lem:h is f-vector}
If $A=S/I$ for $I$ as above, then the Hilbert
series of $A$ is
\[
HS(A,t)=\sum f_{i-1}(\Delta)t^i.
\]
\end{lem}
\begin{proof}
For the exterior Stanley-Reisner ring $E/I_\Delta$, $HS(E/I_\Delta,t)=
\sum f_{i-1}(\Delta)t^i$,  and it is clear that
$HS(A,t)=HS(E/I_\Delta,t)$, with $HF(A,i)=f_{i-1}$.
\end{proof}

\begin{cor}\label{levelCor}
With the notation of Lemma \ref{lem:h is f-vector}, the algebra $A$ is level if and only if the flag complex $\Delta$ is pure. 
\end{cor}

\begin{proof}
Denote by $m_1,\ldots,m_t$ the monomials corresponding to the facets of $\Delta$, that is, each $m_i$ is the product of the variables corresponding to the vertices of the $i$-th facet. Let $I'$ be the annihilator of $m_1,\ldots,m_t$ in the sense of inverse systems, that is, $I'$ consists of the monomials that annihilate each of the monomials $m_1,\ldots,m_t$ under contraction. The minimal generators of $J_{\Delta}$ correspond to minimal non-faces of $\Delta$. Hence, they are in $I'$. Clearly, $J'$ is in $I'$, which gives $I \subset I'$. 

The $(i-1)$-dimensional faces of $\Delta$ correspond to divisors of one of the monomials $m_1,\ldots,m_t$. Thus,  $f_{i-1} (\Delta)$ is equal to  the Hilbert function of $R/I'$ in degree $i$. Hence Lemma \ref{lem:h is f-vector} gives that $A$ and $R/I'$ have the same Hilbert function, which forces $I = I'$. Thus, we have shown that the inverse system of $A$ is generated by the monomials corresponding to the faces of $\Delta$. Now the claim follows. 
\end{proof}

When $A$ is an Artinian quotient by a quadratic monomial ideal $I$, 
then $I$ always has a decomposition as above, and $\Delta$ is flag: it is defined by the non-edges, so is encoded by a graph. Removing a vertex or
an edge from the graph gives rise to a short exact sequence, 
yielding an inductive tool to study WLP for quadratic monomial ideals.
\subsection{Removing an edge or vertex}
Let $e_{ij}$ be the edge corresponding to monomial $x_ix_j$, and 
$v_i$ the vertex corresponding to the variable $x_i$. Write
$S'$ for a quotient of $S$ by some set of variables, which will be
apparent from the context, and $J'$ for the ideal of the squares of
the variables in $S'$. For a face $\sigma \in \Delta$, $\st(\sigma) =
\{ \tau \in \Delta \mid \sigma \subseteq \tau \}$ and $\lk(\sigma) = \partial(\st(\sigma))$.
\begin{lem}\label{sesVertex}
For the short exact sequence
\[
0 \longrightarrow S(-1)/(I:x_i) \stackrel{\cdot x_i}{\longrightarrow} 
S/I \longrightarrow S/(I+x_i) \longrightarrow 0,
\]
we have
\[
\begin{array}{ccc}
S/(I+x_i)  & \simeq & S'/(J'+J_{\Delta'} )\\
S/(I:x_i)  & \simeq & S'/(J'+J_{\Delta''})
\end{array}
\]
where $\Delta' = \Delta \setminus \st(v_i)$, and $\Delta'' = \lk(v_i)$. 
\end{lem}
\begin{proof}
In $(I+x_i)$, all quadrics in $I$ divisible by $x_i$ are
nonminimal, so the remaining quadrics are those not involving
$x_i$, which are exactly the nonfaces of $\Delta \setminus \st(v_i)$.    
\[
\mbox{For }\Delta'',\mbox{ since }
 (J_\Delta+J'):x_i = I + \langle x_i \rangle + \langle  x_k \mid 
x_ix_k \in I \rangle 
\]
it follows that $\Delta''$ is obtained by deleting all vertices not 
connected to $v_i$, as well as $v_i$ itself, so what remains is $\lk(v_i)$. 
\end{proof}
\begin{lem}\label{sesEdge}For the short exact sequence
\[
0 \longrightarrow S(-2)/(I:x_ix_j) \stackrel{\cdot x_ix_j}{\longrightarrow} 
S/I \longrightarrow S/(I+x_ix_j) \longrightarrow 0,
\]

we have
\[
\begin{array}{ccc}
S/(I+x_ix_j)  & \simeq & S'/(J'+J_{\Delta'}) \\
S/(I:x_ix_j)  & \simeq & S'/(J'+J_{\Delta''})
\end{array}
\]
where $\Delta' = \Delta \setminus \st(e_{ij})$, and $\Delta'' = \lk(e_{ij})$. 
\end{lem}
\begin{proof}
Since adding $x_ix_j$ to $I$ corresponds to deleting $e_{ij}$ from
$\Delta$, the description of $\Delta'$ is automatic. For $\Delta''$,
since 
\[
(J_\Delta+J'):x_ix_j = I + \langle x_i, x_j \rangle + \langle  x_k \mid 
x_ix_k \mbox{ or } x_jx_k \in I \rangle 
\]
it follows that $\Delta''$ is obtained by deleting all vertices not 
connected to $e_{ij}$, as well as $e_{ij}$ itself, so what remains is
$\lk(e_{ij})$. 
\end{proof}

\begin{prop}\label{WLPses}
Let 
\begin{itemize}
\item $\Delta'$ denote $\Delta \setminus \st(e_{ij})$ or $\Delta \setminus \st(v_{i})$
\item $\Delta''$ denote $\lk(e_{ij})$ or $\lk(v_{i})$,
\end{itemize}
and write $A'',A,A'$ for the respective Artinian algebras, and $\delta$
for the connecting homomorphism in Equation~\ref{CONN} below. Then 
\begin{itemize}
\item $\mu_i$ is injective if both $\{\mu''_{i-2}$ $($edge$)$  or $\mu''_{i-1}$ $($vertex$)$ $\}$ and $\delta$ are injective. 
\item $\mu_i$ is surjective iff both $\mu'_i$ and $\delta$ are surjective.
\end{itemize}
\end{prop}
\begin{proof}
Apply the snake lemma to the graded pieces
of the short exact sequences of Lemma~\ref{sesVertex} or
Lemma~\ref{sesEdge}. For example,  in the case of Lemma~\ref{sesEdge},
this yields
\[
\xymatrix{ 
0  \ar[r]  & A_{i-2}''\ar[r]^{\cdot x_ix_j}\ar[d]^{\mu''_{i-2}} &A_{i} \ar[d]^{\mu_i}\ar[r]& A_{i}' \ar[r] \ar[d]^{\mu'_i}&0\\ 
0  \ar[r]  & A_{i-1}''\ar[r]^{\cdot x_ix_j}                   &A_{i+1}\ar[r]                  & A_{i+1}' \ar[r]         &0}
\] 
Letting $K'', K, K'$ denote the kernels of $\mu_l$ on the algebras $A'',A,A'$, and similarly
for the cokernels yields a long exact sequence
\begin{equation}\label{CONN}
\xymatrix{ 
0  \ar[r]  & K''_{i-2} \ar[r]^{\cdot x_ix_2} & K_i \ar[r]    & K'_i\ar[dll]^{\delta} & \\
           &   C''_{i-1}  \ar[r]&  C_{i+1}\ar[r]  & C'_{i+1}\ar[r]           & 0},
\end{equation} 
and the result follows.
\end{proof}
\begin{cor} Equation~\ref{CONN} yields simple numerical reasons for failure of WLP for an edge (the vertex case is similar):
\begin{itemize}
\item If $A''$ fails injectivity in degree $i-2$, and $f_{i}(\Delta) \ge f_{i-1}(\Delta)$, then $A$ fails WLP in degree $i$ due to injectivity.
\item If $A'$ fails surjectivity in degree $i$, and $f_{i}(\Delta) \le f_{i-1}(\Delta)$, then $A$ fails WLP in degree $i$ due to surjectivity.
\end{itemize}
\end{cor}
\begin{exm}\label{failfromses}
The ideal $I= \langle x_0^2, x_1^2,x_2^2,x_3^2,x_4^2,x_5^2, x_0x_1, x_2x_3\rangle
\subseteq \KK[x_0,\ldots, x_5]$ has quotient $A=S/I$ with Hilbert series 
$(1, 6, 13, 12,4)$. Applying Proposition~\ref{WLPses} with the role of $\Delta'$ played by Example~\ref{NoWLP} shows that $I$ fails WLP in degree two. This ideal and Examples~\ref{HasWLP},~\ref{NoWLP} are the only quadratic monomial ideals in $\KK[x_0,\ldots, x_5]$ which fail WLP that are not covered by Theorem~\ref{MMROthm} or ~\ref{MMRthm}.
\end{exm}

\section{Topology and $\Z/2$ coefficients } 
\begin{lem}\label{isCC}
If char$(\KK)=2$, $(A,\mu_l)$ is a chain complex.
\end{lem}
\begin{proof}
By Proposition~\ref{sum}, we may assume $l = \sum x_i$, so since
$l^2 = \sum x_i^2 + 2\sum_{i \ne j}x_ix_j$ the result follows. More
is true: the multiplication map sends a monomial $x^\alpha \mapsto \sum_i x^\alpha
x_i$. Dualizing yields the transpose of the simplicial
boundary operator with $\Z/2$ coefficients. 
\end{proof}
When we focus on the WLP in a specific degree $i$, it makes sense
to consider not the simplicial complex corresponding to $J$, but
by an associated simplicial complex:
\begin{defn}\label{trunc}
For the simplicial complex $\Delta$ associated to the squarefree
quadratic monomial ideal $J$, and choice of degree $i$, let
$\Delta(i)$ be the $i$-skeleton of $\Delta$: delete all 
faces of dimension $\ge i+1$ from $\Delta$. Algebraically, this
corresponds to replacing $I$ with $I(i):=I+\langle S_{i+2} \rangle$.
\end{defn}
The next proposition shows the reason to define $\Delta(i)$: over $\Z/2$, it gives 
a precise reason for the failure of WLP in degree $i$.
\begin{prop}\label{TopFail}
Let $A$ be an Artinian quotient by a quadratic monomial ideal
$I$. Then over $\Z/2$, $A$ fails WLP in degree $i$ iff 
\begin{itemize}
\item Surjectivity fails: $H_{i}(\Delta(i), \Z/2) \ne 0$ and $f_{i} \le f_{i-1} $ or 
\item Injectivity fails: $\coker(\partial_i) \ne 0$ and $f_{i-1} \le f_{i} $.
\end{itemize}
\end{prop}
\begin{proof}
Let $A'=S(-1)/(I:l)$, $k = \Z/2$,  and consider the factorization of the four term exact sequence
\[
\xymatrix{
0  \ar[r]  & (\frac{I:l}{I}(-1))_i\ar[r] &A_i \ar[rr]^{\mu_{\ell}}\ar[dr]&& A_{i+1} \ar[r] &(A/l)_{i+1}\ar[r] & 0\\
           &                                       &           &           A'_{i+1}\ar[ur]\ar[dr] &               &           & \\
         &                                       &                 0\ar[ur]        &        &   0            &                & }
\]
WLP fails in degree $i$ iff $f_{i-1} \le f_{i} $ and $\ker(\mu_{\ell}) \ne 0$ or
$f_{i} \le f_{i-1} $ and $\coker(\mu_{\ell}) \ne 0$. By Lemma~\ref{isCC}, over $\Z/2$, 
$\coker(\mu_l) \ne 0$ iff $\ker(\partial_{i}) \ne 0$. 

Because we are considering WLP in degree $i$, replacing $I$ with
$I(i)$ has no impact, and the rank of $\mu_l$ is the same as the 
rank of the dual map of $\Z/2$ vector spaces
\[
A_{i+1}^\vee \stackrel{\partial_i}{\longrightarrow} A_i^\vee.
\]
In particular, $\mu_{\ell}$ is not
surjective iff $\partial_i$ has a kernel iff $H_{i}(\Delta(i), \Z/2)
\ne 0$. Similarly, $\ker(\mu_l) \ne 0$ iff $\coker(\partial_i) \ne 0$.
\end{proof}
\subsection{Lifting to characteristic zero}
Proposition~\ref{TopFail} shows that over $\Z/2$, the failure of WLP for 
quadratic monomial ideals is connected to topology. In
Example~\ref{HasWLP}, $\Delta$ is homotopic to
$S^1$ and 2 points, and in Example~\ref{NoWLP}, $\Delta$ is homotopic to $S^2$.
The Bockstein spectral sequence shows that $\Z/2$ homology abuts to integral
homology; in certain circumstances this impacts the rank of $\mu_l$.
\begin{prop}\label{TopFailLift}
For $A$ an Artinian quotient by a quadratic monomial ideal
$I$ over a field of characteristic zero and degree $i$, consider $\Delta(i)$ of Definition~\ref{trunc}. 
\begin{itemize}
\item If the connecting homomorphism $\delta$ in $(\ref{LES})$ is not injective (and therefore $H_i(\Delta(i), \Z/2) \ne 0$), then $\mu_l$ fails surjectivity; if also $f_i \le f_{i-1}$, WLP fails.
\item If the connecting homomorphism $\delta$ in $(\ref{LES})$ is injective 
(for example if $H_i(\Delta(i), \Z/2) =0$), then $\mu_l$ is surjective and WLP holds.
\end{itemize}
\end{prop}
\begin{proof}
We work with the integral Stanley-Reisner
ring $\Z[x_0,\ldots,x_r]$ and let $B$ be the quotient by the ideal
$I(i)$. The graded pieces of $B$ are free $\Z$-modules, and so we may
tensor the short exact sequence 
\[
0 \longrightarrow \Z \stackrel{\cdot 2}{\longrightarrow} \Z
\longrightarrow \Z/2 \longrightarrow 0
\]
to obtain a commuting diagram with exact rows of $\Z$-modules 
\[
\xymatrix{
0  \ar[r]  & B_{i+1}^\vee\ar[r]^{\cdot 2}\ar[d]^{\mu_l^\vee} &B_{i+1}^\vee \ar[d]^{\mu_l^\vee}\ar[r]& B_{i+1}^\vee \otimes \Z/2 \ar[r] \ar[d]^{\partial_i}&0\\
0  \ar[r]  & B_{i}^\vee\ar[r]^{\cdot 2}                   &B_{i}^\vee\ar[r]                  & B_{i}^\vee \otimes \Z/2 \ar[r]         &0}
\]
The snake lemma yields a long exact sequence 
\begin{equation}\label{LES}
\xymatrix{
0  \ar[r]  & \ker(\mu_l^\vee)\ar[r]^{\cdot 2} &\ker(\mu_l^\vee) \ar[r]    & H_i(\Delta(i), \Z/2)\ar[dll]^{\delta} & \\
           &  \coker(\mu_l^\vee)\ar[r]^{\cdot 2}&  \coker(\mu_l^\vee)\ar[r]  & \coker(\partial_i)\ar[r]           & 0}
\end{equation}
The connecting homomorphism $\delta$ is not injective iff it has a kernel.
Because $\ker(\mu_l^\vee)$ is the kernel of a map between free
$\Z$-modules, it is torsion free and therefore free. As the 
leftmost map in the long exact sequence is induced by $\cdot 2$, 
it follows that $\ker(\mu_l^\vee)$ is a nonzero free $\Z$ module 
iff $\delta$ is not injective. Localizing at $0$ gives an
exact sequence of $\Q$ vector spaces, and so in characteristic zero, 
$A$ fails WLP due to surjectivity in degree $i$ iff  $\delta$ is not an injection
on $H_i(\Delta(i), \Z/2)$ and $f_i \le f_{i-1}$. The proof of the 
second assertion is similar.
\end{proof}
\begin{exm}
When $\Delta$ is a n-cycle, WLP fails if $n$ is even, and holds if $n$ is odd.
To see this, consider the map $\mu_l^\vee$, and $i=1$. Then $\mu_l^\vee$ has the 
form 
\[
\left[ \!
\begin{array}{ccccccc}
1      & 1      & 0     & 0      & \cdots   & 0 & 0\\
0      & 1      & 1     & 0      & \cdots   & 0 & 0\\
0      & 0      & 1     & 1      & \cdots   & 0 & 0\\
\vdots & \ddots & \ddots& \cdots &  \cdots  &  1& 1\\
1      & 0      & \cdots & \cdots & \cdots  & 0 & 1
\end{array}\! \right]
\]
For both cases, $H_1(\Delta(1),\Z/2) \simeq \Z/2$. However, when $n$ is even $\delta$ is zero, whereas
when $n$ is odd $\delta$ is the identity. The 
determinant of $\mu_l^\vee$ is $0$ for $n$ even, and $2$ for
$n$ odd, confirming our calculation.
\end{exm}
\section{Generalizing the Micha\l{}ek Mir\'o-Roig condition} 
There is an interesting connection between Example~\ref{NoWLP}, where
WLP fails in degree two due to topology, and the criterion for failure
of injectivity in degree one appearing in Theorem~\ref{MMRthm}. The key 
observation is that letting $V_1 =\{x_0,x_1\}, V_2 =\{x_2,x_3\}, V_3 =\{x_4,x_5 \}$ be bases for the trio of vector spaces $V_i$, then the 
algebra $A$ in Example~\ref{NoWLP} can be written as
\begin{equation}\label{tensor}
\bigotimes\limits_{i=1}^3 Sym(V_i)/V_i^2 \mbox{ where }V_i \mbox{ is the irrelevant ideal of }Sym(V_i).
\end{equation}
Theorem~\ref{MMRthm} deals with quadratic algebras which have a tensor 
decomposition with only two factors of type $Sym(V_i)/V_i^2$. We show that quadratic algebras with more than two factors of type $Sym(V_i)/V_i^2$ also always fail to have WLP. In fact, we prove an even more general result on tensor product algebras; the next lemma plays a key role:
\begin{lem}{\cite[Lemma 7.8]{BMMRNZ}} \label{BMMNZ lemma}
Let $A = A' \otimes A''$ be a tensor product of two graded Artinian $k$-algebras $A'$ and $A''$. Let $L' \in A'$ and $L'' \in A''$ be linear elements, and set $L := L'+L'' = L' \otimes 1 + 1 \otimes L'' \in A$. Then:

\begin{itemize}

\item[(a)] If the multiplication maps $\times L' : A'_{i-1} \rightarrow A'_i$ and $\times L'' : A''_{j-1} \rightarrow A''_j$ are both not surjective, then the multiplication map
\[
\times L : A_{i+j-1} \rightarrow A_{i+j}
\]
is not surjective.

\item[(b)] If the multiplication maps $\times L' : A'_{i} \rightarrow A'_{i+1}$ and $\times L'' : A''_{j} \rightarrow A''_{j+1}$ are both not injective, then the multiplication map
\[
\times L : A_{i+j} \rightarrow A_{i+j+1}
\]
is not injective.
\end{itemize}
\end{lem}

\begin{thm}\label{thm:tensor}
Let 
\[
A = \bigotimes\limits_{i=1}^n Sym(V_i)/V_i^{k_i}, \mbox{ with }\dim(V_i) \ge 2 \mbox{ and }n, k_i \ge 2.
\]
 Then the algebra $A$ does not have the WLP. 
 \end{thm}
\begin{proof} 
Let $k = \sum_{i=1}^n(k_i-1)$. First, assume  $\dim A_{k-1} \ge \dim A_k$. 
Since $\dim V_i \ge 2$, the multiplication map 
\[
(Sym(V_i)/V_i^{k_i})_{k_i - 2} \stackrel{\cdot \ell}{\longrightarrow} (Sym(V_i)/V_i^{k_i})_{k_i - 1},
\]
is not surjective, where $\ell_i$ is the sum of the variables in $Sym (V_i)$. Hence Lemma \ref{BMMNZ lemma} gives that $\times \ell: \dim A_{k-1} \to \dim A_k$ is not surjective, where $\ell = \ell_1 + \cdots + \ell_n$. Thus,  $A$ does not have the WLP. 

Second, assume $\dim A_{k-1} < \dim A_k$. Consider the residue class $0 \neq \overline{m} \in A_{k-1}$ of $m_1 - m_2 := \ell_1^{k_1 - 2} \ell_2^{k_2 - 1} \cdots \ell_n^{k_n - 1} - \ell_1^{k_1 - 1} \ell_2^{k_2 - 2} \ell_3^{k_3 - 1} \cdots \ell_n^{k_n - 1}$. Then 
\[
\overline{m} \cdot \ell = \overline{m_1} \cdot \ell_1 - \overline{m_2} \cdot \ell_2 = 0. 
\]
Thus, $\mu_\ell: \dim A_{k-1} \to \dim A_k$ is not injective, so  $A$
does not have the WLP.  
\end{proof}
\noindent When we allow one dimensional $V_i$, interesting things can happen:
\begin{thm}
    \label{thm:tensortrivial}
For $A$ Artinian with $A_d \ne 0$ and $A_{d+1}=0$, $C=A \otimes \KK[z]/z^j$ has WLP when $j \ge d+1$.
 \end{thm}
\begin{proof} 
Choose bases $A_0z^i \oplus A_1z^{i-1} \oplus \cdots \oplus A_i$ for $C_i$ and $A_0z^{i+1} \oplus A_1z^{i} \oplus \cdots \oplus A_{i+1}$ for $C_{i+1}$. Then when $i+1 \le d$, the multiplication map is given by a block matrix 
\[
\mu_i^C=
\left[ \!
\begin{array}{ccccc}
I      & 0        & 0     & \cdots   & 0\\
d_0      & I      & 0     & \cdots   & \vdots\\
0      & d_1      & I     & \cdots   & \vdots\\
\vdots & 0        & d_2   & \cdots   & 0     \\
\vdots & \ddots   & 0     & \cdots   &  I \\
0      & 0        & \cdots & \cdots  &  d_i
\end{array}\! \right],
\]
where we write $d_i$ for the multiplication map $\mu_l$ from $A_i
\rightarrow A_{i+1}$; $\mu^C_i$ is clearly injective. When $k = i+1-d >0$, we truncate the last $k$ row blocks and rightmost $k-1$ column blocks of the matrix, and so the matrix still has full rank. 
\end{proof}
\noindent When  $j \le d$, things are more complicated: if $k' = i+1-d >0$, then we truncate
the matrix above by the top $k'$ row blocks, and leftmost $k'-1$ column blocks,
resulting in a matrix of the form
\[
\left[ \!
\begin{array}{ccccc}
d_{k'}  & I      & 0       & \cdots   & \vdots\\
0      & d_{k'+1} & \ddots  & \cdots   & \vdots\\
\vdots & 0        & \ddots & \cdots   & 0     \\
\vdots & \ddots   & 0      & \cdots   &  I \\
0      & 0        & \cdots & \cdots  &  d_i
\end{array}\! \right],
\]
and WLP depends on properties of $A$. We now return to the quadratic case.
\subsection{Tensor algebras of quadratic quotients}
\begin{cor}\label{dsquared}
For an Artinian $A$ which is a quotient by quadratic monomials with $A_{i+1} \ne 0$, $C=A \otimes \KK[z]/z^2$ has WLP in degree $i$ iff $\cdot \ell^2$ has full rank on $A_{i-1}$, where $\ell$ is the sum of the variables of $A$.
\end{cor}
\begin{proof}
Choose bases $A_{i-1}z \oplus A_{i}$ for $C_i$ and $ A_iz \oplus A_{i+1}$ for $C_{i+1}$, and let $I_j$ be the identity on $A_j$. If $A_{i+1} \ne 0$, then $C_i \stackrel{\mu_{\ell}}{\rightarrow} C_{i+1}$ is given by the block matrix 
\[
\left[ \!
\begin{array}{cc}
d_{i-1}     & I_i\\
0        &d_{i}
\end{array}\! \right], 
\]
and changing basis for the row and column space reduces the matrix to 
\[
\left[ \!
\begin{array}{cc}
d_id_{i-1}&0\\
0        &I_i
\end{array}\! \right]. 
\]
As $d_i\cdot d_{i-1} = \mu_{l^2}$,
the result follows. 
\end{proof}
\begin{cor}
If $A_2 \ne 0$ and char$(\KK)$=2, then $A\otimes \KK[z]/z^2$ does not have WLP.
\end{cor}
\begin{proof} By Lemma~\ref{isCC} $\mu_l$ is a differential so
  $d_id_{i-1}=0$; apply Corollary~\ref{dsquared}.
\end{proof}
\begin{prop}\label{onetrivialfactor}
$(\Sym(V)/V^2) \bigotimes\limits_{i=1}^n \KK[z_i]/z_i^2$ has WLP if char$(\KK) \ne 2$. 
\end{prop}
\begin{proof}
Let $A = \bigotimes\limits_{i=1}^n \KK[z_i]/z_i^2$, $C=Sym(V)/V^2$ and 
$B=C \otimes A$. By \cite{Stan}, $A$ has WLP; choose a basis for 
$B_i \simeq (C_1 \otimes A_{i-1}) \oplus A_i$ respecting the 
direct sum, and similarly for $B_{i+1}$. 
Let $d_i^B$ denote the multiplication map $B_i \rightarrow B_{i+1}$,
and let $d_i^A$ and $I^A_i$ denote the multiplication and identity maps 
on $A_i$. Then 
\begin{equation}\label{TensorMatrix}
d_i^B =
\left[ \!
\begin{array}{ccccc}
d_{i-1}^A & 0          & \hdots & 0             & I^A_i\\
0              & \ddots & \ddots & \vdots    & \vdots\\
\vdots     & \ddots & \ddots & 0            & \vdots\\
0             & \hdots & 0          &d_{i-1}^A & I^A_i\\
0             & \hdots & \hdots & 0             & d_i^A 
\end{array}\! \right],
\end{equation}
is a $\dim(B_{i+1}) \times \dim(B_i)$ matrix; note that if $i$ is $0$ or $n$ then $d_i^B$ changes shape, but has full rank at these values. When $n$ is odd, the graded components of $A$ of maximal dimension occur in degrees $\frac{n-1}{2}$ and $\frac{n+1}{2}$, while if $n$ is even there is a unique maximal graded component of $A$ in degree $\frac{n}{2}$. 

If $\dim(A_{i-1}) \le \dim(A_i) \le \dim(A_{i+1})$ then since $A$ has WLP both $d_{i-1}^A$ and $d^A_i$ are injective, hence so is $d^B_i$, and similarly for surjectivity. So it suffices to study the behavior of $B$ at the peak; a check shows that $\dim(B_i)$ is maximal at $m= \lceil \frac{n+1}{2}\rceil$, with 
\[
\dim(B_{m-1}) < \dim(B_m) > \dim(B_{m+1}),
\]
hence we need only show that $B_{m-1} \hookrightarrow B_m$ and 
$ B_m \twoheadrightarrow B_{m+1}$. 

When $n$ is odd, this is automatic: there are always two consecutive degrees 
in $A$ where $d_{i-1}^A$ and $d_i^A$ are both injective  or both 
surjective, and from the structure of $d_i^B$ above, this means $d_i^B$
is also injective or surjective. 

When $n$ is even, the result is more delicate. Let $d=d_{\frac{n}{2}-1}^A$ and $I =
I_{A_{\frac{n}{2}-1}}$. Then because $A$ is Gorenstein, the matrices for the multiplication 
maps are symmetric. The easiest way to see this is to note that these maps are the maps on the Koszul complex of 
$\Lambda(\KK^n)$, but with all signs positive: if char$(\KK) = 2$ then $A \simeq\Lambda(\KK^n)$ . So if $n$ is even,
we may write the matrix of Equation~\ref{TensorMatrix} as 
\begin{equation}\label{TensorMatrix2}
d_i^B =
\left[ \!
\begin{array}{ccccc}
d & 0          & \hdots & 0             & I\\
0              & \ddots & \ddots & \vdots    & I\\
\vdots     & \ddots & \ddots & 0            & \vdots\\
0             & \hdots & 0          &d & I\\
0             & \hdots & \hdots & 0             & d^T,
\end{array}\! \right],
\end{equation}
and it suffices to show that 
\[
\phi = \left[ \!
\begin{array}{cc}
d & I\\
0 & d^T 
\end{array}\! \right],
\]
has full rank. Since $d$ is injective, ordering bases so 
the top left $\dim(A_{m-1}) \times \dim(A_{m-1})$ block of $d$ has full rank, we also have that the bottom right $\dim(A_{m-1}) \times \dim(A_{m-1})$ block 
has full rank, so that $\phi$ has rank $\dim(A_m)+\dim(A_{m-1})$ and therefore $d_i^B$ has full rank. 
\end{proof}

\begin{prop}\label{twotrivialfactor}
If char($\KK$) $\ne 2$ and $\dim(V_1)=a \ge 2 \le b = \dim(V_2)$, then
\[
\Sym(V_1)/V_1^2 \bigotimes \Sym(V_2)/V_2^2 \bigotimes\limits_{i=1}^n \KK[z_i]/z_i^2
\]
has WLP iff $n$ is odd.
\end{prop}
\begin{proof}
We first show that if $n=2k$ is even then the algebra does not have the WLP. Theorem \ref{MMRthm} (or Theorem \ref{thm:tensor}, taking $k_1 = k_2 = 2$ and $n=2$) shows that 
\[
\Sym(V_1) /V_1^2 \bigotimes \Sym(V_2)/V_2^2
\]
fails both injectivity and surjectivity from degree 1 to degree 2. On the other hand, $\bigotimes_{i=1}^n \KK[z_i]/z_i^2$ is a complete intersection of quadrics, so it reaches its unique peak in degree $k$. Thus it fails injectivity from degree $k$ to $k+1$ and fails surjectivity from degree $k-1$ to $k$. Then Lemma \ref{BMMNZ lemma} shows that our tensor product fails both injectivity and surjectivity from degree $k+1$ to $k+2$, so WLP fails.

Now, let 
\[
A = \bigotimes\limits_{i=1}^n \KK[z_i]/z_i^2 \mbox{  with $n$ odd, } C=\Sym(V_1)/V_1^2 \otimes \Sym(V_2)/V_2^2,
\]
and write $B=C \otimes A$. 
Choose bases for $B_i \simeq (C_2 \otimes A_{i-2}) \oplus (C_1 \otimes A_{i-1}) \oplus A_i$ and for $B_{i+1} \simeq (C_2 \otimes A_{i-1}) \oplus (C_1 \otimes A_i)  \oplus A_{i+1}$ respecting the direct sums. Let $d_i^A$ and $I^A_i$ denote, respectively,
the multiplication and identity maps on $A_i$, with similar notation for $B$ and $C$. Then with respect to the bases above, 
\begin{equation}\label{twotrivdiff}
d_i^B =
\left[ \!
\begin{array}{ccc}
I^C_2 \otimes d_{i-2}^A & d^C_1 \otimes I^A_{i-1}       & 0\\
0                                  & I^C_1 \otimes d^A_{i-1}       & d^C_0 \otimes I^A_i\\
0                                   &                           0             & d_i^A 
\end{array}\! \right],
\end{equation}
As in the proof of Proposition~\ref{onetrivialfactor}, the key will be to use the fact that $A$ has WLP and is a complete intersection. From the structure of $d_i^B$, if we
choose $i$ such that 
\[
\dim(A_{i-2}) \le \dim(A_{i-1}) \le \dim(A_{i})\le \dim(A_{i+1}). 
\]
then $d^B_i$ is injective; and similarly for surjectivity. Because $n$ is 
odd, by symmetry the only problematic case is when $i=\frac{n+1}{2}$, so
that  
\begin{equation}\label{threeinarow}
\dim(A_{i-2}) < \dim(A_{i-1}) = \dim(A_{i}) > \dim(A_{i+1}). 
\end{equation}
Using that $a,b \ge 2$, a dimension count shows we need
to prove that $d_i^B$ is injective. Since $A$ has WLP, the 
map $d^A_{i-1}$ is an isomorphism, and as in 
the proof of Proposition~\ref{onetrivialfactor}, $d^A_{i}$ is the
transpose of $d^A_{i-2}$. However, in contrast to 
Proposition~\ref{onetrivialfactor}, the map $d^C_1$ plays a role.

If $\{x_1, \ldots, x_a \}$ is a basis for $V_1$ and $\{y_1, \ldots y_b\}$ a basis for $V_2$ then $d_1^C(\sum x_i - \sum y_j) = 0$, and an Euler characteristic 
computation shows that this is the only element of the kernel of
$d^C_1$. This means that with respect to the ordered bases 
$\{x_1, \ldots, x_a, y_1, \ldots, y_{b-1}, \sum x_i- \sum y_j\}$ 
for $C_1$ and $\{x_1y_1, \ldots, x_1y_b, x_2y_1, \ldots x_2y_b, \ldots x_ay_b \}$  for $C_2$, the matrix $d^C_1 \otimes I^A_{i-1}$ has a block decomposition with
rightmost $\dim(C_2)\cdot\dim(A_{i-1}) \times \dim(A_{i-1})$ submatrix zero. 
For example, when $a=2$ and $b=3$, with $I= I^A_{i-1}$ the matrix takes the form
\begin{equation}\label{Cdiff}
\left[ \!
\begin{array}{ccccc}
I & 0 & I & 0 & 0\\
I & 0 & 0 & I & 0\\
I & 0 & 0 & 0 & 0\\
0 &I & I &   0 & 0\\
0 &I & 0 &  I &  0\\
0 &I & 0 & 0 & 0
\end{array}\! \right],
\end{equation}
The block of $d_i^B$
corresponding to the submatrix $d^C_0 \otimes I^A_i$
consists of $\dim(C_1)$ stacked copies of $I^A_i = I^A_{i-1}$, 
while the block of $d_i^B$ corresponding to $I^C_1 \otimes d^A_{i-1} $
consists of   $\dim(C_1)$ diagonal copies of the invertible matrix
$U$ representing $d^A_{i-1}$; this follows because $n$ is odd so 
$\dim(A_{i-1}) = \dim(A_i)$ and $A$ has WLP.

Continuing with the $a=2$, $b=3$ example, and writing 
$d = d^A_{i-2}$, $I = I^A_i = I^A_{i-1}$ and $U$ for the
invertible matrix $d^A_{i-1}$, the matrix representing $d_i^B$ is
\begin{equation}\label{Bigdiff}
\left[ \!
\begin{array}{cccccccccccc}
d & 0 &0 & 0 & 0 &0 &I & 0 & I & 0 & 0 &0 \\
0 & d &0 & 0 & 0 &0 &I & 0 & 0 & I & 0&0\\
0 & 0 &d & 0 & 0 &0 &I & 0 & 0 & 0 & 0&0\\
0 & 0 &0 & d & 0 &0 &0 &I & I &   0 & 0&0\\
0 & 0 &0 & 0 & d &0 &0 &I & 0 &  I &  0&0\\
0 & 0 &0 & 0 & 0 &d &0 &I & 0 & 0 & 0&0\\
0 & 0 &0 & 0 & 0 &0 &U & 0 & 0 & 0 & 0 &I \\
0 & 0 &0 & 0 & 0 &0 &0 & U & 0 & 0 & 0 &I \\
0 & 0 &0 & 0 & 0 &0 &0 & 0 & U & 0 & 0 &I \\
0 & 0 &0 & 0 & 0 &0 &0 & 0 & 0 & U & 0 &I \\
0 & 0 &0 & 0 & 0 &0 &0 & 0 & 0 & 0 & U &I \\
0 & 0 &0 & 0 & 0 &0 &0 & 0 & 0 & 0 & 0 &d^T \\
\end{array}\! \right],
\end{equation}
Now, since $U$ is invertible of the same size as $I$,
we may use it to row reduce the matrix above. Because
the bottom row block of $d^C_1 \otimes I^A_{i-1}$ has 
only one nonzero entry, this allows row reduction of
the matrix for $d_i^B$ in the example to 
\begin{equation}\label{Bigdiff}
\left[ \!
\begin{array}{cccccccccccc}
d & 0 &0 & 0 & 0 &0 &0 & 0 & 0 & 0 & 0 &-2U^{-1} \\
0 & d &0 & 0 & 0 &0 &0 & 0 & 0 & 0 & 0&-2U^{-1}\\
0 & 0 &d & 0 & 0 &0 &0 & 0 & 0 & 0 & 0&-2U^{-1}\\
0 & 0 &0 & d & 0 &0 &0 &0 & 0 &  0 & 0&-2U^{-1}\\
0 & 0 &0 & 0 & d &0 &0 &0 & 0 &  0 &  0&-2U^{-1}\\
0 & 0 &0 & 0 & 0 &d &0 &0 & 0 & 0 & 0&-1U^{-1}\\
0 & 0 &0 & 0 & 0 &d &I & 0 & 0 & 0 & 0 &0 \\
0 & 0 &0 & 0 & 0 &d &0 & I & 0 & 0 & 0 &0 \\
0 & 0 &0 & 0 & 0 &d &0 & 0 & I & 0 & 0 &0 \\
0 & 0 &0 & 0 & 0 &d &0 & 0 & 0 & I & 0 &0 \\
0 & 0 &0 & 0 & 0 &d &0 & 0 & 0 & 0 & I &0 \\
0 & 0 &0 & 0 & 0 &0 &0 & 0 & 0 & 0 & 0 &d^T \\
\end{array}\! \right],
\end{equation}
So we need only show that 
\begin{equation}\label{tinydiff}
\left[ \!
\begin{array}{cc}
d &-U^{-1}\\
0       &d^T
\end{array}\! \right],
\end{equation}
has full rank, which follows as in Proposition~\ref{onetrivialfactor}.
\end{proof}

\medskip

\begin{thm}~\label{Ttensor}
If $\hbox{char} (\KK) \neq 2$, and 
\[
C = \bigotimes_{i=1}^n Sym(V_i)/V_i^2 \hbox{ with } \dim(V_i) = d_i, 
\]
then C has WLP if and only if one of the following holds:
\begin{enumerate}
\item $d_2,...,d_n$ are 1;
\item $d_3,...,d_n$ are 1, and $n$ is odd.
\end{enumerate}
\end{thm}
\begin{proof}
From our work above, we have that:
\begin{enumerate}
\item WLP holds by Proposition~\ref{onetrivialfactor}.
\item WLP holds by Proposition~\ref{twotrivialfactor}.
\end{enumerate}
Proposition~\ref{twotrivialfactor} also covers failure of WLP  in case (2) when $n$ is even, so what remains is to show WLP fails when three or more $d_i \ge 2$; if no $d_i=1$, this follows from Theorem~\ref{thm:tensor}.

Assume for convenience that $d_1 \geq d_2 \geq \dots \geq d_n$. We first consider the case $d_{n-1} > d_n = 1$; of course then we have $n-1 \geq 3$ by assumption. Let
\[
A' = \bigotimes_{i=1}^{n-2} Sym(V_i)/V_i^2 \ \ \hbox{ , } \ \ A'' = Sym(V_{n-1}) \otimes \KK[z]/z^2 \ \ , \ \ C = A' \otimes A''.
\]
By the proof of Theorem \ref{thm:tensor}, $A'$ fails both injectivity and surjectivity from degree $n-3$ to degree $n-2$. The Hilbert function of $A''$ is $(1,d_{n-1}+1, d_{n-1})$ so clearly it fails surjectivity from degree 0 to degree 1, and it fails injectivity from degree 1 to degree 2. Then by Lemma \ref{BMMNZ lemma}, $C$ fails surjectivity from degree $n-2$ to degree $n-1$, and it fails injectivity in the same degree, hence it fails to have WLP.

Now let 
\[
B' = \bigotimes_{i=1}^m Sym(V_i)/V_i^2 \ \ , \ \ B'' = \bigotimes_{i=1}^p \KK [z_i]/z_i^2 \ \ , \ \ C = B' \otimes B''
\]
where $\dim V_i \geq 2$ for $1 \leq i \leq m$, $m \geq 3$, and  $ m+p = n$. We will show that $C$ fails WLP. 

\medskip

\noindent \underline{Case 1}: $p = 2k$ is even, $k \geq 1$.

\medskip

We will show that $C$ fails WLP from degree $m-1+k$ to $m+k$. By Theorem \ref{thm:tensor}, we know that the multiplication on $B'$ from degree ${m-1}$ to $m$ is neither surjective nor injective. Since $B''$ is a complete intersection of quadrics with even socle degree $2k$, it fails injectivity from degree $k$ to $k+1$, and it fails surjectivity from degree $k-1$ to $k$. Then Lemma \ref{BMMNZ lemma} applies as before to show that $C$ fails WLP as claimed.

\medskip

\noindent \underline{Case 2}: $p = 2k+1$ is odd.

\medskip

We will again show that $C$ fails WLP from degree $m+k-1$ to $m+k$.  Our proof is by induction on $k$, having just shown the case $k=0$. So assume $k \geq 1$. 
We rewrite $C$ as follows:
\[
C' = \bigotimes_{i=1}^m Sym(V_i)/V_i^2 \otimes \bigotimes_{i=1}^{2k-1} \KK[z_i]/z_i^2 \ \ , \ \ C'' = \KK[x,y]/(x^2,y^2) \ \ , \ \ C = C' \otimes C''.
\]
By induction, $C'$ fails both injectivity and surjectivity from degree $m+k-2$ to $m+k-1$. Clearly $C''$ fails injectivity from degree 1 to degree 2, and fails surjectivity from degree 0 to degree 1. Then Lemma \ref{BMMNZ lemma} again gives the result.
\end{proof}

\begin{rmk} If $K_{a,b}$ is the complete bipartite graph
and $K_{a,b,r}$ the cone over $K_{a,b,r-1}$ with
$K_{a,b,0} = K_{a,b}$, then the Stanley-Reisner ring of
$K_{a,b,r}$ plus squares of variables has WLP iff $r$ is odd.  
\end{rmk}

\begin{rmk} We close with some directions for future work. 
\newline
\begin{itemize}
\item Quadratic monomial ideals are always Koszul but may fail to have
  WLP. In \cite{MN2}, Migliore-Nagel conjectured
that all quadratic Artinian Gorenstein algebras have WLP; this was
recently disproved by Gondim-Zappala \cite{GZ}. It should be interesting
to investigate the confluence between Koszulness, quadratic Gorenstein
algebras, and WLP.
\item Use the syzygy bundle techniques of \cite{BK}, combined with
  Hochster's and Laksov's work on syzygies \cite{HL} to study WLP. This will mean
studying the WLP for products of linear forms, and we are at work \cite{MNS} on this.
\item Connect to the work of Singh-Walther \cite{SW} on the Bockstein spectral sequence and local cohomology of Stanley-Reisner rings.
\end{itemize}
\end{rmk}
\noindent{\bf Acknowledgments} Computations were performed using 
Macaulay2, by Grayson and Stillman, available at: {\tt http://www.math.uiuc.edu/Macaulay2/}.
Our collaboration began at the BIRS workshop ``Artinian algebras and
Lefschetz properties'', organized by S. Faridi, A. Iarrobino,
R. Mir\'o-Roig, L. Smith and J. Watanabe, and we thank them and BIRS
for a wonderful and stimulating environment. The first author was partially funded by the Simons Foundation under Grant \#309556, the second author was partially supported  by the Simons Foundation under grant \#317096.



\bibliographystyle{amsalpha}

\end{document}